\documentclass[12pt,a4paper]{amsart}

\newtheorem{theo+}              {Theorem}           [section]
\newtheorem{prop+}  [theo+]     {Proposition}
\newtheorem{coro+}  [theo+]     {Corollary}
\newtheorem{lemm+}  [theo+]     {Lemma}
\newtheorem{exam+}  [theo+]     {Example}
\newtheorem{rema+}  [theo+]     {Remark}
\newtheorem{defi+}  [theo+]     {Definition}
\def \r{\mbox{${\mathbb R}$}}

\newenvironment{theorem}{\begin{theo+}}{\end{theo+}}
\newenvironment{proposition}{\begin{prop+}}{\end{prop+}}
\newenvironment{corollary}{\begin{coro+}}{\end{coro+}}
\newenvironment{lemma}{\begin{lemm+}}{\end{lemm+}}

\usepackage{amsthm}
\theoremstyle{plain} \theoremstyle{remark}
\newtheorem{remark}{Remark}
\newtheorem{example}{Example}

\def\E{/\kern-1.0em \equiv }

\author{Ze-Ping Wang }
\author{Ye-Lin Ou$^{*}$ }
\author{Han-Chun Yang$^{**}$ }
\address{Department of Mathematics, \newline\indent Yunnan University,\newline\indent 
Kunming 650091, P. R. China
\newline\indent E-mail:zeping.wang@gmail.com \;(Wang)\\\newline\indent E-mail: hyang@ynu.edu.cn\; (H. Yang)\\\newline\indent
\newline\indent Department of
Mathematics,\newline\indent Texas A $\&$ M
University-Commerce,\newline\indent Commerce, TX 75429
USA.\newline\indent E-mail:yelin.ou@tamuc.edu\; (Ou)}
\thanks{*Research supported by NSF of Guangxi (P. R. China), 2011GXNSFA018127.\\
\indent **Research supported by the NSF of China (11361073)}
\begin{document}
\title[Biharmonic maps from a 2-sphere]{Biharmonic maps from a 2-sphere}
\date{12/08/2013}
\maketitle
\section*{Abstract}
\begin{quote}
{\footnotesize Motivated by the rich theory of harmonic maps from a $2$-sphere, we study biharmonic maps from a $2$-sphere in this paper. We first derive biharmonic equation for rotationally symmetric maps between rotationally symmetric $2$-manifolds. We then apply the equation to obtain a classification of biharmonic maps in a family of rotationally symmetric maps between $2$-spheres. We also find many examples of proper biharmonic maps defined locally on a $2$-sphere. Our results seem to suggest that any biharmonic map $S^2\longrightarrow (N^n, h)$ be a weakly conformal immersion.}
\end{quote}
\section{Introduction}
In this paper, we work on the category of smooth objects, so all manifolds, tensor fields, and maps, etc. are assumed to be smooth.\\

{\em A harmonic map} is a map $\phi: (M,g)\longrightarrow (N,h)$ between Riemannian manifolds that is a critical point of the energy functional defined by
\begin{equation}\notag
E(\phi)=\frac{1}{2}\int_\Omega |{\rm d}\phi|^2v_g,
\end{equation}
where $\Omega$ is a compact domain of $ M$. The Euler-Lagrange
equation of the energy functional gives the harmonic map equation ( \cite{ES})
\begin{equation}\label{fhm}
\tau(\phi) \equiv {\rm Tr}_g\nabla\,d \phi=0,
\end{equation}
where $\tau(\phi)$ is called the tension field of
the map $\phi$.\\

{\em A biharmonic map} is a map $\phi: (M,g)\longrightarrow (N,h)$ between Riemannian manifolds that is a critical point of the bienergy functional defined by 
\begin{equation}\notag
E_{2}(\phi)=\frac{1}{2}\int_\Omega |\tau(\phi)|^2v_g,
\end{equation}
where $\Omega$ is a compact domain of $ M$. The
Euler-Lagrange equation of this functional gives the biharmonic map
equation (\cite{Ji1})
\begin{equation}\label{BTF}
\tau_{2}(\phi):={\rm
Trace}_{g}(\nabla^{\phi}\nabla^{\phi}-\nabla^{\phi}_{\nabla^{M}})\tau(\phi)
- {\rm Trace}_{g} R^{N}({\rm d}\phi, \tau(\phi)){\rm d}\phi =0,
\end{equation}
where $R^{N}$ denotes the curvature operator of $(N, h)$ defined by
$$R^{N}(X,Y)Z=
[\nabla^{N}_{X},\nabla^{N}_{Y}]Z-\nabla^{N}_{[X,Y]}Z.$$

Clearly, any harmonic map ($\tau(\phi)\equiv 0$) is always a biharmonic map. We call a biharmonic map that is not harmonic a {\em proper biharmonic map}. \\

The study of biharmonic maps (as a special case of k-polyharmonic maps with $k=2$) was proposed by Eells-Lemaire in \cite{EL} (Section (8.7)). Jiang \cite{Ji1}, \cite{Ji2}, \cite{Ji3} made a first effort to study such maps by calculating the first and second variational formulas of the bienergy functional and specializing on the biharmonic isometric immersions which nowadays are called biharmonic submanifolds. Very interestingly, the notion of biharmonic submanifolds was also introduced by B. Y. Chen \cite{Ch} in a different way in his study of the finite type submanifolds in Euclidean spaces. Since 2000, the study of biharmonic maps has been attracting a growing attention and it has become an active area of research with many progresses. We refer the readers to \cite{BK}, \cite{BFO2}, \cite{BMO2}, \cite{LOn1}, \cite{MO}, \cite{NUG}, \cite{Ou1}, \cite{Ou4}, \cite{OL}, \cite{Oua}, and the references therein for some recent geometric study of general biharmonic maps. For some recent progress on biharmonic submanifolds see \cite{CI}, \cite{Ji2}, \cite{Ji3}, \cite{Di}, \cite{CMO1}, \cite{CMO2}, \cite{BMO1}, \cite{BMO3}, \cite{Ou3}, \cite{OT}, \cite{OW}, \cite{NU}, \cite{TO}, \cite{CM}, \cite{AGR} and the references therein. For biharmonic conformal immersions and submersions see \cite{Ou2}, \cite{Ou5}, \cite{BFO1}, \cite{LO}, \cite{WO} and the references therein.\\

In this paper, we study biharmonic maps from a $2$-sphere. One of our motivations comes from the following observations. There are many interesting examples and a rich theory of harmonic maps from a $2$-sphere:
\begin{itemize}
\item Chern-Goldberg \cite{CG}: Any harmonic immersion $f:S^2\longrightarrow (N^n,h)$ has to be minimal, or equivalently, a conformal immersion;
\item Sacks-Uhlenbeck \cite{SU} and Wood \cite{Wo1}: Any harmonic map $f:S^2\longrightarrow (N^n,h)$ with $n\ge 3$ has to be a conformal branched minimal immersion;
\item Smith \cite{Sm}: Any homotopy class of maps $S^2\longrightarrow S^2$ has a harmonic map representative;
\item There exist harmonic embeddings of $S^2$ into $S^3$, equipped with arbitrary metric (\cite{SM1});
\item There are many beautiful explicit constructions that can be used to produce harmonic maps from $S^2$ into projective spaces, Grassmannian manifolds, and Lie groups (See e.g., Uhlenbeck \cite{Uh}, Burstall-Salamon \cite{BS}, Burstall-Wood \cite{BW} and Wood \cite{Wo2}, \cite{Wo3});
\item Fern$\acute{\rm a}$ndez \cite{Fe}:  The dimension of the space of harmonic maps from the $2$-sphere to the $2n$-sphere is $2d + n^2$. There is an explicit algebraic method to construct all harmonic maps from the $2$-sphere to the $n$-sphere.
\end{itemize} 

It would be interesting to know if any of the above results can be generalized to the case of proper biharmonic maps.
Knowing that the only known example of proper biharmonic maps from $S^2$ is the biharmonic isometric
immersion $S^2(\frac{1}{\sqrt{2}})\longrightarrow S^3$ \cite{CMO1} (or a composition of this with a totally geodesic maps from $S^3$ into another manifold (See, e.g., \cite{Ou1})) we would especially like to know the answer to the question: {\em Does there exist a proper biharmonic map $\varphi:S^2\longrightarrow (N^n,h)$ that is NOT a conformal immersion?}\\

In this paper, we study biharmonicity of rotationally symmetric maps from $S^2$. We obtain a classification of biharmonic maps in a family of rotationally symmetric  maps between $2$-spheres. We are able to find many examples of locally defined proper biharmonic maps from $S^2$. Very interestingly, we find that none of these locally defined proper biharmonic maps allows an extension to a biharmonic map defined globally on $S^2$. Our results seem to suggest the following \\

{\em Conjecture:} any biharmonic map $\varphi: S^2\longrightarrow (N^n,h)$ is a weakly conformal immersion.
\section{Biharmonic equations for rotationally symmetric maps}
In this section, we will derive biharmonic equations for a large class of maps that includes rotationally symmetric maps between rotationally symmetric manifolds. We will need the following lemma that gives the equation of biharmonic maps in local coordinates.
\begin{lemma}\label{Owww}\cite{OL}
Let $\phi :(M^{m}, g)\longrightarrow (N^{n}, h)$ be a map between
Riemannian manifolds with $\phi (x^{1},\ldots, x^{m})=(\phi^{1}(x),
\ldots, \phi^{n}(x))$ with respect to local coordinates $(x^{i})$ in
$M$ and $(y^{\alpha})$ in $N$. Then, $\phi$ is biharmonic if and
only if it is a solution of the following system of PDE's
\begin{eqnarray}\notag\label{BI3}
&& \Delta\tau^{\sigma} +2g(\nabla\tau^{\alpha},
\nabla\phi^{\beta}) {\bar \Gamma_{\alpha\beta}^{\sigma}}
+\tau^{\alpha}\Delta \phi ^{\beta}{\bar
\Gamma_{\alpha\beta}^{\sigma}}\\ && +\tau^{\alpha}
g(\nabla\phi^{\beta}, \nabla\phi^{\rho})(\partial_{\rho}{\bar
\Gamma_{\alpha\beta}^{\sigma}}+{\bar
\Gamma_{\alpha\beta}^{\nu}}{\bar \Gamma_{\nu\rho}^{\sigma}})
-\tau^{\nu}g(\nabla\phi^{\alpha}, \nabla\phi^{\beta}){\bar
R}_{\beta\,\alpha \nu}^{\sigma}=0,\\\notag && \sigma=1,\, 2,\,
\ldots, n,
\end{eqnarray}
where $\tau^1, \ldots, \tau^n$ are components of the tension field
of the map $\phi$, $\nabla,\;\Delta$ denote the gradient and the
Laplace operators defined by the metric $g$, and
${\bar\Gamma_{\alpha\beta}^{\sigma}}$ and ${\bar R}_{\beta\,\alpha
\nu}^{\sigma}$ are the components of the connection and the
curvature of the target manifold.
\end{lemma}
\begin{lemma}\label{PLast}
The map $\varphi:(M^2,dr^2+\sigma^2(r)d\theta^2)\longrightarrow
(N^2,d\rho^2+\lambda^2(\rho)d\phi^2)$ with $\varphi(r,\theta)=(\rho(r), cr+k\theta+a_{2})$
is biharmonic if and only if it solves the system
\begin{equation}\label{Sl1}
\begin{cases}
x''+\frac{\sigma'}{\sigma}x'-(c^2+\frac{\kappa^2}{\sigma^2})(\lambda\lambda'(\rho))'(\rho)x-(2cy'+y^2)\lambda\lambda'(\rho)
=0,\\
cy''+yy'+\frac{2c^2\lambda'(\rho)}{\lambda}x'+
2c^2\left(\frac{\sigma'\lambda'(\rho)}{\sigma\lambda}
+\frac{\rho'(\lambda\lambda'(\rho))'(\rho)}{\lambda^2}\right)x=0,\\
x=\tau^1=\rho''+\frac{\sigma'}{\sigma}\rho'-(c^2+\frac{\kappa^2}{\sigma^2})\lambda\lambda'(\rho),\\
y=\tau^2=\frac{2c\rho'\lambda'(\rho)}{\lambda}+\frac{c\sigma'}{\sigma},
\end{cases}
\end{equation}
\end{lemma}
\begin{proof}
One can easily compute the connection coefficients of the domain and
the target surfaces to get
\begin{eqnarray}\notag
&& \Gamma^1_{11}=0, \hskip0.3cm\Gamma^1_{12}=0, \hskip0.3cm
\Gamma^1_{22}=-\sigma\sigma',\;\;\; \Gamma^2_{11}=0, \hskip0.3cm
\Gamma^2_{12}=\frac{\sigma'}{\sigma}, \hskip0.3cm \Gamma^2_{22}=0,\\\notag
&&\bar{\Gamma}^1_{11}=0, \hskip0.3cm\bar{\Gamma}^1_{12}=0,
\hskip0.3cm \bar{\Gamma}^1_{22}=-\lambda\lambda'(\rho),\;\;\;
\bar{\Gamma}^2_{11}=0, \hskip0.3cm \bar{\Gamma}^2_{12}=\frac{\lambda'(\rho)}{\lambda},
\hskip0.3cm \bar{\Gamma}^2_{22}=0.
\end{eqnarray}
We can also check that the components of the Riemannian curvature of
the target surface are given by
\begin{eqnarray}\notag
&&\bar{R}^1_{212}=-\lambda\lambda''(\rho),\; \bar{R}^1_{221}=\lambda\lambda''(\rho),\;
\bar{R}^2_{112}=\frac{\lambda''(\rho)}{\lambda}, \;\bar{R}^2_{121}=-\frac{\lambda''(\rho)}{\lambda}\\\notag && {\rm
others}\; \quad \bar{R}^l_{kij}=0,
\end{eqnarray}
and that the tension field of the map $\varphi$ has components
\begin{eqnarray}\label{Sl2}
\tau^1 &=& g^{ij}(\varphi^{1}_{ij}-\Gamma^k_{ij}\varphi^{1}_{k}+\bar{\Gamma}^1_{\alpha\beta}\varphi^{\alpha}_{i}\varphi^{\beta}_{j})
=\rho''+\frac{\sigma'}{\sigma}\rho'-(c^2+\frac{k^2}{\sigma^2})\lambda\lambda'(\rho),\\
\tau^2&=&g^{ij}(\varphi^{2}_{ij}-\Gamma^k_{ij}\varphi^{2}_{k}+\bar{\Gamma}^2_{\alpha\beta}\varphi^{\alpha}_{i}\varphi^{\beta}_{j})
=\frac{2c\rho'\lambda'(\rho)}{\lambda}+\frac{c\sigma'}{\sigma}.
\end{eqnarray}
Using the notations $x=\tau^1$ and $y=\tau^2$
and performing a further computation we have
\begin{eqnarray}\label{121}
&& \Delta\tau^1=g^{ij}(\tau^1_{ij}-\Gamma^k_{ij}\tau^1_k)=x''+\frac{\sigma'}{\sigma}x',\\
&& \Delta\tau^2=g^{ij}(\tau^2_{ij}-\Gamma^k_{ij}\tau^2_k)=y''+\frac{\sigma'}{\sigma}y',\\
&& 2g(\nabla\tau^{\alpha},\nabla\varphi^{\beta})\bar{\Gamma}^1_{\alpha\beta}=-2c\lambda\lambda'(\rho)y',\\
&& 2g(\nabla\tau^{\alpha},\nabla\varphi^{\beta})\bar{\Gamma}^2_{\alpha\beta}=\frac{2(cx'+\rho'y')\lambda'(\rho)}{\lambda},\\
&& \tau^{\alpha}\Delta\varphi^{\beta}\bar{\Gamma}^1_{\alpha\beta}=-\frac{c\sigma'\lambda\lambda'(\rho)}{\sigma}y,\\
&& \tau^{\alpha}\Delta\varphi^{\beta}\bar{\Gamma}^2_{\alpha\beta}=\frac{\lambda'(\rho)}{\lambda}[\frac{c\sigma'}{\sigma}x+(\rho''+\frac{\sigma'}{\sigma}\rho')y],\\
&&\tau^{\alpha} g(\nabla\varphi^{\beta}, \nabla\varphi^{\rho})\partial_{\rho}{\bar
\Gamma_{\alpha\beta}^{1}}=-c\rho'(\lambda\lambda'(\rho))'(\rho)y,\\
&& \tau^{\alpha}g(\nabla\varphi^{\beta},\nabla\varphi^{\rho})\bar{\Gamma}^v_{\alpha\beta}\bar{\Gamma}^1_{v\rho}=-(c^2+\frac{k^2}{\sigma^2})\lambda'^2(\rho)x-c\rho'\lambda'^2(\rho)y,\\
&& \tau^{\alpha}g(\nabla\varphi^{\beta}, \nabla\varphi^{\rho})\partial_{\rho}{\bar
\Gamma_{\alpha\beta}^{2}}=c\rho'(\frac{\lambda'(\rho)}{\lambda})'(\rho)x+\rho'^2(\frac{\lambda'(\rho)}{\lambda}))'(\rho)y,
\end{eqnarray}
\begin{eqnarray}
&& \tau^{\alpha}g(\nabla\varphi^{\beta},\nabla\varphi^{\rho})\bar{\Gamma}^v_{\alpha\beta}\bar{\Gamma}^2_{v\rho}=c\rho'(\frac{\lambda'(\rho)}{\lambda})^2 x+\rho'^2(\frac{\lambda'(\rho)}{\lambda})^2 y-(c^2+\frac{k^2}{\sigma^2})\lambda'^2(\rho)y,\\
&& -\tau^vg(\nabla\varphi^{\alpha},\nabla\varphi^{\beta})\bar{R}^1_{\beta\alpha
v}=-(c^2+\frac{k^2}{\sigma^2})\lambda\lambda''(\rho)x+c\rho'\lambda\lambda''(\rho)y,
\end{eqnarray}
and
\begin{eqnarray}\label{131}
&&-\tau^vg(\nabla\varphi^{\alpha},\nabla\varphi^{\beta})\bar{R}^2_{\beta\alpha
v}=\frac{c\rho'\lambda''(\rho)}{\lambda}x-\frac{\rho'^2\lambda''(\rho)}{\lambda}y.
\end{eqnarray}
Substituting (\ref{121})$\sim$ (\ref{131}) into (\ref{BI3}) we conclude that the map $\varphi$ is biharmonic if and only if
\begin{equation}\notag
\begin{cases}
x''+\frac{\sigma'}{\sigma}x'-(c^2+\frac{k^2}{\sigma^2})(\lambda\lambda'(\rho))'(\rho)x-2c\lambda\lambda'(\rho)y'
-c \left (\frac{\sigma'\lambda\lambda'(\rho)}{\sigma}
+2\rho'\lambda'^2(\rho)\right)y=0,\\
y''+(\frac{\sigma'}{\sigma}+\frac{2\rho'\lambda'(\rho)}{\lambda})y'
+\left(\frac{\lambda'(\rho)}{\lambda}(\rho''+\frac{\sigma'}{\sigma}\rho')-(c^2+\frac{k^2}{\sigma^2})\lambda'^2(\rho)
\right)y\\
+\frac{2c\lambda'(\rho)}{\lambda}x'+
c\left(\frac{\sigma'\lambda'(\rho)}{\sigma\lambda}
+\frac{2\rho'\lambda''(\rho)}{\lambda}\right)x=0,\\
x=\tau^1=\rho''+\frac{\sigma'}{\sigma}\rho'-(c^2+\frac{k^2}{\sigma^2})\lambda\lambda'(\rho),\\
y=\tau^2=\frac{2c\rho'\lambda'(\rho)}{\lambda}+\frac{c\sigma'}{\sigma},
\end{cases}
\end{equation} 
which is equivalent to the system (\ref{Sl1}). Thus we we obtain the lemma.
\end{proof}
\begin{remark} (i) With $\sigma=1, c=0, k=1$ our Lemma \ref{PLast} recovers Proposition 3.1 in \cite{OL}; (ii) One can also check that our Lemma \ref{PLast} also recovers the case for $m=n=2$ in Theorem 5.4 in \cite{BMO2}.
\end{remark}

Some other straightforward applications of Lemma \ref{PLast} can be stated as
\begin{corollary}\label{Pc1}
The map $\varphi:(M^2,dr^2+\sigma^2(r)d\theta^2)\longrightarrow
(N^2,d\rho^2+\lambda^2(\rho)d\phi^2)$, $\varphi(r,\theta)=(\rho(r), k\theta+a_{2})$
is biharmonic if and only if it solves the system
\begin{equation}\label{Sl4}
\begin{cases}
x''+\frac{\sigma'}{\sigma}x'-\frac{k^2(\lambda\lambda')'(\rho)}{\sigma^2}x=0,\\
x=\tau^1=\rho''+\frac{\sigma'}{\sigma}\rho'-\frac{k^2\lambda\lambda'(\rho)}{\sigma^2}.
\end{cases}
\end{equation}

In particular, the rotationally symmetric map $\varphi:(T^2, dr^2+d\theta^2) \longrightarrow (S^2,h=d\rho^2+\sin^2 \rho d\phi^2)$ from a flat torus into a sphere with $\varphi(r,\theta)=(\rho(r),\kappa\theta)$ is biharmonic if and only if
\begin{equation}\label{rsyTS2}
\rho^{(4)}-2\kappa^2\cos(2\rho)\rho''+2\kappa^2\sin(2\rho)\rho'^2+\frac{\kappa^4}{4}\sin(4\rho)=0.
\end{equation}
\end{corollary}
\begin{remark}

Note that Equation (\ref{rsyTS2}) was obtained in \cite{MR} by using a 1-dimensional variational approach. It was also observed in \cite{MR} that when $\rho=\frac{\pi}{4}, \frac{3\pi}{4}$,
$\varphi$ give proper biharmonic maps.
\end{remark}
\begin{corollary}\label{Ps1}
The rotationally symmetric map $\varphi:(S^2,dr^2+\sin^{2}rd\theta^2)\longrightarrow
(S^2,d\rho^2+\sin^2\rho d\phi^2)$ with $\varphi(r,\theta)=(\rho(r), k\theta)$
is biharmonic if and only if the function $\rho=\rho(r)$ solves the system
\begin{equation}\label{Ps2}
\begin{cases}
x''+\cot{r} x'-\frac{k^2\cos2\rho}{\sin^{2}r}x=0,\\
x=\tau^1=\rho''+\cot r\rho'-\frac{k^2\sin2\rho}{2\sin^{2}r},
\end{cases}
\end{equation}
or equivalently, 
\begin{equation}\label{ys1}
\begin{cases}
\frac{d^2x}{d t^2}=k^2x\cos (2\rho),\\
x(t)=\cosh^2t\left(\frac{d^2\rho}{d t^2}-k^2\sin\rho\cos \rho\right),
\end{cases}
\end{equation}
where $t=\ln|\tan\frac{r}{2}|$.
\end{corollary}
\begin{proof}
Equation (\ref{Ps2}) is obtained by applying Corollary \ref{Pc1} with $\sigma=\sin r, \lambda=\sin\rho, a_2=0$ whilst Equation (\ref{ys1}) comes from Equation (\ref{Ps2}) by a transformation $t=\ln|\tan \frac{r}{2}|$. 
\end{proof}
\section{A classification of biharmonic maps between 2-spheres}
Note that there are many smooth maps between spheres. For example, the following family of maps was studied in Peng-Tang \cite{PT}.
Let $f_k : S^n \longrightarrow S^n\; (k > 0)$ be defined by 
\begin{equation}\label{FM}
(\cos r, \sin r \cdot X) \longrightarrow (\cos (kr), \sin (k r) \cdot X),
\end{equation}
where $0\le r \le \pi $ and $ X\in S^{n-1}\subset \r^n$. It was proved in \cite{PT} that $f_k$ is a k-form and the Brouwer degree of $f_k$ is\\
${\rm deg} f_k=
\begin{cases}
k,\;\; if\; $n$\; is\; odd;\\
1\;\; if\; $n$\; is\; even \;and\; $k$\; is\; odd;\\
0\;\; otherwise.
\end{cases}$\\

Note also that with respect to geodesic polar coordinates, this family of maps can be described as $f_k(r,\theta)= (kr, \theta)$, a family of rotationally symmetric maps between spheres. \\

It is easy to check that the quadratic polynomial map $F:\mathbb{R}^{3}\longrightarrow \mathbb{R}^{3}$ defined by $F(x, y, z)=(x^{2}-y^{2}-z^{2}, \;2xy,\; 2xz)$ restricts to a map between spheres $f=F|_{S^{2}}:S^{2}\longrightarrow S^{2}$. Using geodesic polar coordinates $(r, \theta)$ on the domain sphere and $(\rho, \phi)$ on the target sphere we can have a local expression of the map given by 
\begin{equation}
f:(S^2, {\rm d}\,r^{2}+\sin ^{2} r\,{\rm d}\,\theta^{2})\longrightarrow (S^2, {\rm d}\,\rho^{2}+\sin ^{2} \rho\,{\rm d}\,\phi^{2}),\;\;f(r,\theta)=(2r, \theta).
\end{equation}

It follows that this restriction of the polynomial map is a rotationally symmetric map between $2$-spheres belonging to the family (\ref{FM}): $f_k: S^2\longrightarrow S^2$. One can further check that the tension field of the map $f$ is $\tau(f)=2 \sin 2r\frac{\partial}{\partial r}$, so it is NOT a harmonic map. It would be interesting to know whether there exists any proper biharmonic map in the family of rotationally symmetric maps $f_k$. \\

Our next theorem gives a classification of biharmonic maps in a class of maps $S^2\longrightarrow S^2$ which includes the family $f_k$ of rotationally symmetric maps as a subset.
\begin{theorem}\label{ow111}
The rotationally symmetric map $\varphi:(S^2,dr^2+\sin^2rd\theta^2)\longrightarrow
(S^2,d\rho^2+\sin^2\rho d\phi^2)$ with $\varphi(r,\theta)=(ar+a_1, k\theta)$ and $a\neq 0$
is biharmonic if and only if $a^2=1, k^2=1$, and $a_1=0$, or, $a_1=\pi $, i.e., the map is actually harmonic.
\end{theorem}
\begin{proof}
For $\rho=ar+a_1,$ it follows from Corollary \ref{Ps1} that $\varphi$ is biharmonic if and only if it solves the system
\begin{equation}\label{ow41}
\begin{cases}
x''+\cot rx'-\frac{k^2\cos2\rho}{\sin^2r}x=0,\\
x=\tau^1=a\cot r-\frac{k^2\sin2\rho}{2\sin^2r},\\
\rho=ar+a_1.
\end{cases}
\end{equation}
A straightforward computation using the last two equations of (\ref{ow41}) gives
\begin{equation}\label{ow410}
\begin{cases}
x'=-\frac{a}{\sin^2r}-\frac{k^2a\sin r\cos2\rho-k^2\cos r\sin2\rho}{\sin^3r},\\
x''=\frac{2a\cos r}{\sin^3r}-\frac{(k^2-2k^2a^2)\sin^2 r\sin2\rho+3k^2\cos^2 r\sin2\rho-4k^2a\sin r\cos r\cos2\rho}{\sin^4r}.
\end{cases}
\end{equation}
Substituting (\ref{ow410}) into Equation $(\ref{ow41})$ we have
\begin{equation}\notag
\begin{cases}
\frac{2a\sin r\cos r+2(2k^2a^2-k^2)\sin^2 r\sin2\rho-4k^2\cos^2 r\sin2\rho+4k^2a\sin r\cos r\cos2\rho+k^4\sin2\rho\cos2\rho}{2\sin^4r}=0,\\
\rho=ar+a_1,
\end{cases}
\end{equation}
which is equivalent to
\begin{equation}\notag
\begin{cases}
2a\sin r\cos r+2(2k^2a^2-k^2)\sin^2 r\sin2\rho
-4k^2\cos^2 r\sin2\rho\\+4k^2a\sin r\cos r\cos2\rho+k^4\sin2\rho\cos2\rho=0,\\
\rho=ar+a_1.
\end{cases}
\end{equation}
By a further computation, we can rewrite the above equation as
\begin{equation}\label{ow0}
\begin{cases}
a\sin 2r+(2k^2a^2-3k^2)\sin2\rho
-(2k^2a^2+k^2)\cos2 r\sin2\rho\\+2k^2a\sin 2r\cos2\rho+k^4\sin2\rho\cos2\rho=0,\\
\rho=ar+a_1.
\end{cases}
\end{equation}
Write $f(r)=a\sin 2r+(2k^2a^2-3k^2)\sin2\rho-(2k^2a^2+k^2)\cos2 r\sin2\rho\\+2k^2a\sin 2r\cos2\rho+k^4\sin2\rho\cos2\rho$, then we have
\begin{equation}\notag
\begin{array}{lll}
f'(r)=2a\cos 2r+2a(2k^2a^2-3k^2)\cos2\rho+2k^2\sin 2 r\sin2\rho\\
+(2k^2a-4k^2a^3)\cos2 r\cos2\rho+2k^4a\cos4\rho,
\end{array}
\end{equation}
\begin{equation}\notag
\begin{array}{lll}
f''(r)=-4a\sin 2r-4a^2(2k^2a^2-3k^2)\sin2\rho
\\+(8k^2a^4-4k^2a^2+4k^2)\cos2 r\sin2\rho
+8k^2a^3\sin2 r\cos2\rho
-8k^4a^2\sin4\rho,
\end{array}
\end{equation}
and
\begin{equation}\notag
\begin{array}{lll}
f'''(r)=-8a\cos 2r-8a^3(2k^2a^2-3k^2)\cos2\rho\\
+(-32k^2a^4+8k^2a^2-8k^2)\sin2 r\sin2\rho\\
+(16k^2a^5+8k^2a^3+8k^2a)\cos2 r\cos2\rho
-32k^4a^3\cos4\rho.
\end{array}
\end{equation}
Equation $(\ref{ow0})$ implies that for any $r$, we have
\begin{equation}\label{ow05}
\begin{cases}
f(r)=0,\\
f'(r)=0,\\
f''(r)=0,\\
f'''(r)=0,\\
\rho=ar+a_1.\\
\end{cases}
\end{equation}

Since $a\neq 0$, one can easily check that if $k= 0$, then Equation (\ref{ow41} ) has no solution. So from now on we assume that $k\ne 0$. Substituting $r_0=\frac{\pi}{2},\;\rho_0=a\frac{\pi}{2}+a_1$ into
Equation (\ref{ow05}) we have
\begin{equation}\label{ow06}
\begin{cases}
f(r_0)=2k^2\sin\rho_0\cos\rho_0(4a^2-2+k^2\cos2\rho_0)=0,\\
f'(r_0)=2a\left\{-1-k^4+(4k^2a^2-4k^2)\cos2\rho_0+2k^4\cos^22\rho_0\right\}=
0,\\
f''(r_0)=8k^2\sin\rho_0\cos\rho_0\left(-4a^4+4a^2-1-4k^2a^2\cos2\rho_0\right)=0,\\
f'''(r_0)=8a\left\{1+(-4k^2a^4+2k^2a^2-k^2)\cos2\rho_0-8k^4a^2\cos^22\rho_0+4k^4a^2\right\}=0,\\
2\rho_0=a\pi+2a_1.
\end{cases}
\end{equation}

Noting that $r\in(0,\pi)$ and $\rho(r) \in(0,\pi)$ we conclude that $k\sin \rho_0\ne 0$. We will solve Equation $(\ref{ow06})$ by
the following two cases:\\
{\bf Case $(i)$:} $\cos\rho_0\neq0$.\\
In this case, Equation $(\ref{ow06})$ becomes
\begin{equation}\label{ow073}
\begin{cases}
4a^2-2+k^2\cos2\rho_0=0,\\
-1-k^4+(4k^2a^2-4k^2)\cos2\rho_0+2k^4\cos^22\rho_0=
0,\\
-4a^4+4a^2-1-4k^2a^2\cos2\rho_0=0,\\
1+(-4k^2a^4+2k^2a^2-k^2)\cos2\rho_0-8k^4a^2\cos^22\rho_0+4k^4a^2=0,\\
2\rho_0=a\pi+2a_1.
\end{cases}
\end{equation}
By substituting the first equation of $(\ref{ow073})$ into the second , the third, and the fourth we have
\begin{equation}\label{ow083}
\begin{cases}
k^4=16a^4-8a^2-1,\\
12a^4-4a^2-1=0,\\
-112a^6+112a^4-24a^2-1+4k^4a^2=0.
\end{cases}
\end{equation}
Solving the second equation (\ref{ow083}) we have $a^2=\frac{1}{2}$. Substituting $a^2=\frac{1}{2}$ into the first equation of (\ref{ow083}) yields $k^4=-1$, which shows Equation (\ref{ow083}) and hence (\ref{ow073}) has no real solution in this case.\\
{\bf Case $(ii)$:} $\cos\rho_0=0$,\;i.e.,\;$\rho_0=\frac{a\pi}{2}+a_1=\frac{\pi}{2}$.\\
In this case, Equation $(\ref{ow06})$ reduces to
\begin{equation}\notag
\begin{cases}
-1+k^4-4k^2a^2+4k^2=0,\\
1+k^2+4k^2a^4-2k^2a^2-4k^4a^2=0,
\end{cases}
\end{equation}
which is equivalent to
\begin{equation}\label{ow008}
\begin{cases}
4k^2(a^2-1)=k^4-1,\\
4k^2a^2(a^2-1)=4k^4a^2-2k^2a^2-k^2-1.
\end{cases}
\end{equation}

We can easily check that if $a^2=1$, then Equation (\ref{ow008}) has a solution $k^2=1$. In this case, we have $a_1=0$ or $a_1=\pi$ and hence $\tau^1=a\cot r-\frac{k^2\sin2\rho}{2\sin^2r}=0$. This implies that the map $\varphi:(S^2,dr^2+\sin^2rd\theta^2)\longrightarrow
(S^2,d\rho^2+\sin^2\rho d\phi^2)$ with $\varphi(r,\theta)=( r, k \theta) \; {\rm or}\; \varphi(r,\theta)=(- r+\pi, k \theta)$ is a harmonic map.\\
If $a^2\neq1$, then $k^2\neq1$. In this case, we solve Equation (\ref{ow008}) for $a^2$ in terms of $k$ to have 
\begin{equation}\label{ow010}
a^2=\frac{k^2+1}{3k^4-2k^2+1}.
\end{equation}
Substituting (\ref{ow010}) into the first equation of $(\ref{ow008})$ we have
\begin{equation}\label{ow011}
3k^6+13k^4-k^2+1=0.
\end{equation}
Setting $k^2=t$, the above equation becomes
\begin{equation}\label{ow012}
3t^3+13t^2-t+1=0.
\end{equation}

Now we consider the function $\phi(t)=3t^3+13t^2-t+1$ defined on the interval $[0,+\infty)$. It is an elementary exercise to check that the absolute minimum value of this function over the interval $[0, \infty)$ is $\phi(\frac{-13+\sqrt{178}}{9})=\frac{4}{243}(1247-89\sqrt{178})>0$. It follows that Equation (\ref{ow012}) has no positive solution. This implies that Equation (\ref{ow011}) has no real solution and hence Equation (\ref{ow008}) has no solution in this case.\\

Summarizing the results in Cases (i) and (ii) we obtain the theorem.
\end{proof}
\begin{corollary}
The globally defined smooth map $f=F|_{S^{2}}:S^{2}\longrightarrow S^{2}$ obtained from the restriction of the polynomial map $F:\mathbb{R}^{3}\longrightarrow \mathbb{R}^{3}$, $F(x, y, z)=(x^{2}-y^{2}-z^{2}, \;2xy,\; 2xz)$, is neither a harmonic nor a biharmonic map.
\end{corollary}
\begin{proof}
As we mentioned at the beginning of the section, the map $f$ is a rotationally symmetric map with $f(r, \theta)=(2r, \theta)$. So, by our Theorem \ref{ow111}, the map $f$ is not a biharmonic map. Substituting $a=2, k=1$ into the second equation of (\ref{ow41}) we obtain the first component of the tension field $\tau^1=2\sin 2r\ne 0$. It follows that the map $f$ is neither a harmonic map.
\end{proof}
\section {Locally defined biharmonic maps from a 2-sphere}
We first prove the following proposition which shows that for a special class of maps between rotationally symmetric manifolds, the biharmonic map equation reduces to biharmonic function equation. This will be used to construct many locally defined proper biharmonic maps from a $2$-sphere into itself.
\begin{proposition}\label{Pc21}
The map $\varphi:(M^2,dr^2+\sigma^2(r)d\theta^2)\longrightarrow
(N^2,d\rho^2+\lambda^2(\rho)d\phi^2)$ with $\varphi(r,\theta)=(\rho(r), a_{2})$
is biharmonic if and only if $\Delta_M^2\rho=0$, i.e., $\rho(r)$ is a biharmonic function on $(M^2,dr^2+\sigma^2(r)d\theta^2)$, which can be determined by the integral
\begin{eqnarray}\label{Pc2}
&&\rho(r)=\int\left\{\frac{\int\left(C_1\sigma(r)\int\frac{{\rm d}r}{\sigma(r)}+C_2\sigma(r)\right){\rm d}r+C_3}{\sigma(r)}\right\}{\rm d}r+C_4
,\\\notag
\end{eqnarray}
where $C_{1},\,C_{2},\,C_{4}\,$ and $C_{4}\,$ are constants. Furthermore, the map is proper biharmonic if $C_1^2+C_2^2\ne 0$.
\end{proposition}
\begin{proof}
Using Corollary \ref{Pc1} , it follows that $\varphi:(M^2,dr^2+\sigma^2(r)d\theta^2)\longrightarrow
(N^2,d\rho^2+\lambda^2(\rho)d\phi^2)$ with $\varphi(r,\theta)=(\rho(r), a_{2})$ is biharmonic if and only if it solves the system
\begin{equation}\label{c1}
\begin{cases}
x''+\frac{\sigma'}{\sigma}x'=0,\\
x=\tau^1=\rho''+\frac{\sigma'}{\sigma}\rho'.
\end{cases}
\end{equation}
Noting that $x=\rho''+\frac{\sigma'}{\sigma}\rho'=\Delta_M \rho$ and $\Delta_M x=x''+\frac{\sigma'}{\sigma}x'$ we conclude that Equation (\ref{c1}) is equivalent to $\Delta_M^2\rho=0$. This gives the first statement of the proposition.\\

To solve Equation (\ref{c1}),  we integrate the first equation to have
\begin{eqnarray}\notag
&&x=C_1\int\frac{{\rm d}r}{\sigma(r)}+C_2.\\\notag
\end{eqnarray}
Substituting this into the second equation of (\ref{c1}) we have
\begin{eqnarray}\notag
&&\rho''+\frac{\sigma'}{\sigma}\rho'=C_1\int\frac{{\rm d}r}{\sigma(r)}+C_2,
\end{eqnarray}
which is solved by
\begin{eqnarray}\notag
\rho(r)=\int\left\{\frac{\int\left(C_1\sigma(r)\int\frac{{\rm d}r}{\sigma(r)}+C_2\sigma(r)\right){\rm d}r+C_3}{\sigma(r)}\right\}{\rm d}r+C_4,
\end{eqnarray}
\end{proof}
where $C_{1},\,C_{2},\,C_{4}\,$ and $C_{4}\,$ are constant. Thus, we obtain the proposition.
\begin{remark} 
Applying Proposition \ref{Pc21} we can conclude that the map \\$\varphi: (\r^2\setminus\{0\}, dr^2+r^2d\theta^2)\longrightarrow (N^2,d\rho^2+\lambda^2(\rho)d\phi^2)$ with $\varphi(r,\theta)=(\rho(r), a_{2})$ is biharmonic if and only if $\rho(r)=c_1r ^2\ln r +c_2r^2+c_3 \ln r+c_4$, which is a result in (b) of Proposition 5.5 (for $m=n=2$) in \cite{BMO2}.
\end{remark}
As another straightforward application of Proposition \ref{Pc21}, we have the following corollary which gives many locally defined proper biharmonic maps between $2$-spheres.
\begin{corollary}\label{S2}
For constants $C_1, C_2, C_3, C_4$ with $C_1^2+C_2^2\ne 0$, the rotationally symmetric map $\varphi:(S^2,dr^2+\sin^2rd\theta^2)\longrightarrow
(S^2,d\rho^2+\sin^2\rho \,d\phi^2)$ with $\varphi(r,\theta)=(\rho(r), a_{2})$
is a proper biharmonic if 
\begin{eqnarray}\label{Pc2}
\rho(r)=\int\left\{\frac{\int\left(C_1\sin r\int\frac{{\rm d}r}{\sin r}+C_2\sin r\right){\rm d}r+C_3}{\sin r}\right\}{\rm d}r+C_4.
\end{eqnarray}
\end{corollary}
\begin{remark}
We would like to point out that Corollary \ref{S2} provides many examples of locally defined proper biharmonic maps between two $2$-spheres. However, none of them can be extended to a globally defined map $S^2\longrightarrow S^2$. This can be seen from the fact that each of the maps provided by Corollary \ref{S2} is determined by a locally defined biharmonic function on $S^2$. No locally defined biharmonic function can be extended to the whole sphere $S^2$ as it is well known that any globally defined biharmonic function on $S^2$ has to be a constant.
\end{remark}

In the rest of this section, we will show that the equations for biharmonic maps from $S^2$ into some special choices of rotationally symmetric manifolds can be solved completely. First, let us prove the following lemma.
\begin{lemma}\label{Lm10}
Let $\lambda^2(\rho)=A\rho^2+2C_{0}\rho+C>0$, and $C_0,A, C, k, a_2 $ be constants. Then, the map $\varphi:(M^2,dr^2+\sigma^2(r)d\theta^2)\longrightarrow
(N^2,d\rho^2+\lambda^2(\rho)d\phi^2)$ defined by $\varphi(r,\theta)=(\rho(r), k\theta+a_{2})$
is biharmonic if and only if
\begin{equation}\label{pl31}
\begin{cases}
x''+\frac{\sigma'}{\sigma}x'-\frac{k^2A}{\sigma^2}x=0,\\
x=\tau^1=\rho''+\frac{\sigma'}{\sigma}\rho'-\frac{k^2(A\rho+C_0)}{\sigma^2}.
\end{cases}
\end{equation}
\end{lemma}
\begin{proof}
For $\lambda^2(\rho)=A\rho^2+2C_{0}\rho+C>0$, we have $\lambda\lambda'(\rho)=A\rho+C_{0}$ and $(\lambda\lambda'(\rho))'(\rho)=A$.
By Corollary \ref{Pc1} , the map $\varphi:(M^2,dr^2+\sigma^2(r)d\theta^2)\longrightarrow
(N^2,d\rho^2+\lambda^2(\rho)d\phi^2)$ with $\varphi(r,\theta)=(\rho(r), k \theta+a_{2})$ is biharmonic if and only if it solves the system
\begin{equation}\label{ppl31}
\begin{cases}
x''+\frac{\sigma'}{\sigma}x'-\frac{k^2A}{\sigma^2}x=0,\\
x=\tau^1=\rho''+\frac{\sigma'}{\sigma}\rho'-\frac{k^2(A\rho+C_0)}{\sigma^2}.
\end{cases}
\end{equation}
Thus, we obtain the lemma.
\end{proof}
\begin{theorem}\label{pb}
The rotationally symmetric map $\varphi:(S^2, dr^2+\sin^2(r) d\theta^2)\longrightarrow
(\r^2, d\rho^2+(\rho+1)d \phi^2)$ with $\varphi(r,\theta)=(\frac{1}{4}(\ln \tan \frac{r}{2})^2-\ln \sin r +1, \;\theta)$
is a proper biharmonic map.
\end{theorem}
\begin{proof} First, we prove the following\\
{\bf Claim:} Let $C_0, C $ be constants so that $\lambda^2(\rho)=2C_{0}\rho+C>0$. Then, the map $\varphi:(M^2,dr^2+\sigma^2(r)d\theta^2)\longrightarrow
(N^2,d\rho^2+\lambda^2(\rho)d\phi^2)$ with $\varphi(r,\theta)=(\rho(r), k\theta+a_{2})$
is biharmonic if and only if
\begin{eqnarray}\label{p11}
\rho(r)=\int\left\{\frac{\int\left(C_1\sigma(r)\int\frac{{\rm d}r}{\sigma(r)}+C_2\sigma(r)+\frac{k^2C_0}{\sigma(r)}\right){\rm d}r+C_3}{\sigma(r)}\right\}{\rm d}r+C_4,
\end{eqnarray}
where $C_{1},\,C_{2},\,C_{3}\,$ and $C_{4}\,$ are constants. \\
{\bf Proof of the Claim:} For $\lambda^2(\rho)=2C_{0}\rho+C$, we apply Lemma \ref{Lm10} with $A=0$ to conclude that the map $\varphi:(M^2,dr^2+\sigma^2(r)d\theta^2)\longrightarrow
(N^2,d\rho^2+\lambda^2(\rho)d\phi^2)$ with $\varphi(r,\theta)=(\rho(r), k\theta+a_{2})$
is biharmonic if and only if it solves the system
\begin{equation}\label{p31}
\begin{cases}
x''+\frac{\sigma'}{\sigma}x'=0,\\
x=\tau^1=\rho''+\frac{\sigma'}{\sigma}\rho'-\frac{k^2C_0}{\sigma^2}.
\end{cases}
\end{equation}
Integrating the first equation of (\ref{p31}) we obtain 
\begin{equation}\notag
x=C_1\int \frac{1}{\sigma(r)} dr+C_2.
\end{equation}
Substituting this into the second equation of (\ref{p31}) and multiplying $\sigma(r)$ to both sides of the resulting equation we have
\begin{equation}
(\rho'\sigma)'=C_1\sigma \int \frac{1}{\sigma(r)} dr+C_2\sigma +\frac{k^2C_0}{\sigma}.
\end{equation}

Integrating this second order ODE we obtain the Claim. \\

To prove the theorem, we first notice that the biharmonic maps given in the Claim are proper biharmonic maps for $C_1^2+C_2^2\ne 0$ since the first component of the tension field is $\tau^1=x=C_1\int \frac{1}{\sigma(r)} dr+C_2$. Now, we apply the Claim with $\sigma (r) =\sin r,\; \lambda^2(\rho)=\rho+1$ and $C=C_2=C_4=k=1, C_0=1/2, C_1=C_3=0$ to conclude that the rotationally symmetric map $\varphi:(S^2, dr^2+\sin^2(r) d\theta^2)\longrightarrow
(\r^2, d\rho^2+(\rho+1)d \phi^2)$ with $\varphi(r,\theta)=(\rho(r), \;\theta)$
is a proper biharmonic map if and only if 
\begin{eqnarray}\label{p12}
\rho(r)=\int\left\{\frac{\int\left(\sin r+\frac{1}{2\sin r}\right){\rm d}r}{\sin r}\right\}{\rm d}r+1.
\end{eqnarray}
A further integration of the above integral gives the required result. Thus, we complete the proof of the theorem.
\end{proof}
\begin{remark} 
We would like to point out that the proper biharmonic map\\ $\varphi:(S^2, dr^2+\sin^2(r) d\theta^2)\longrightarrow
(\r^2, d\rho^2+(\rho+1)d \phi^2)$ with\\ $\varphi(r,\theta)=(\frac{1}{4}(\ln \tan \frac{r}{2})^2-\ln \sin r +1, \;\theta)$ is actually defined on the sphere with two points (the north and the south poles) deleted since $r\ne 0, \pi$. 
\end{remark}

\begin{theorem}\label{mv}
Let $\lambda^2(\rho)=\rho^2+2C_0\rho+C>0$, and $ C_0,\; C$ be constants. Then, the map $\varphi:(S^2,dr^2+\sin^2(r)d\theta^2)\longrightarrow
(N^2,d\rho^2+\lambda^2(\rho)d\phi^2)$ with $\varphi(r,\theta)=(\rho(r), \theta)$ is biharmonic if and only if 
\begin{eqnarray}\notag
\rho(r)&=&(C_1-C_2+C_3)|\cot \frac{r}{2}|+(2C_1\ln|\tan \frac{r}{2}|+C_4)|\tan \frac{r}{2}|\\\label{GD26}&&-(C_1|\tan \frac{r}{2}|+C_2|\cot \frac{r}{2}|)\ln (1+\tan^2 \frac{r}{2})-C_0,
\end{eqnarray}
where $ C_1,\;C_2,\;C_3,\;C_4$ are constants. Furthermore, when $C_1^2+C_2^2\neq0$, the rotationally symmetric maps determined by (\ref{GD26}) are proper biharmonic maps.
\end{theorem}
\begin{proof}
Using Lemma \ref{Lm10} with $\sigma(r)=\sin r, A=1, k=1,\; {\rm and}\; a_2=0$ we conclude that $\varphi$ is biharmonic if and only if it solves the system
\begin{equation}\label{mv1}
\begin{cases}
x''+\cot rx'-\frac{1}{\sin^2r}x=0,\\
x=\tau^1=\rho''+\cot r\rho'-\frac{1}{\sin^2r}\rho-\frac{C_0}{\sin^2r}.
\end{cases}
\end{equation}
To solve this system of ODEs we introduce new variable by letting $t=\ln|\tan \frac{r}{2}|$. It follows that 
\begin{equation}\label{GD11}
\begin{cases}
\rho' =\frac{d\rho}{d r}=\frac{d\rho}{d t}\frac{d t}{d r}=\frac{1}{\sin r}\frac{d\rho}{d t},\\
\rho'' =\frac{d^2\rho}{d r^2}=\frac{1}{\sin^2 r}\frac{d^2\rho}{d t^2}-\frac{\cos r}{\sin^2 r}\frac{d\rho}{d t},\\
x' =\frac{dx}{d r}=\frac{1}{\sin r}\frac{dx}{d t},\\
x'' =\frac{d^x}{d r^2}=\frac{1}{\sin^2 r}\frac{d^2x}{d t^2}-\frac{\cos r}{\sin^2 r}\frac{dx}{d t}.
\end{cases}
\end{equation}
Substituting $\sin r=\frac{2\tan \frac{r}{2}}{1+\tan^2 \frac{r}{2}}=\frac{2e^t}{1+e^{2t}}$ and Equation (\ref{GD11}) into (\ref{mv1}) we have 
\begin{equation}\label{GD12}
\begin{cases}
\frac{d^2x}{d t^2}- x=0,\\
x=\frac{(1+e^{2t})^2}{4e^{2t}}\left(\frac{d^2\rho}{d t^2}-\rho-C_0\right).
\end{cases}
\end{equation}
It is very easy to see that the general solution of the first equation of (\ref{GD12}) as
\begin{equation}
x=C_1e^{-t}+C_2e^{t}.
\end{equation}
Substituting this into the second equation of (\ref{GD12}) we obtain
\begin{equation}\label{GD10}
\frac{d^2\rho}{d t^2}-\rho=\frac{4e^{2t}}{(1+e^{2t})^2}(C_1e^{-t}+C_2e^{t})+C_0.
\end{equation}
Using the method of variation of parameters we can have the general solution of (\ref{GD10}) as
\begin{eqnarray}\label{VP}
\rho(t)=C_3e^{-t}+C_4e^{t}+u_1(t)e^{-t}+u_2(t)e^{t}
\end{eqnarray}
where the parameters $u_1, u_2$ are determined by
\begin{eqnarray}\notag
u'_1(t) &=&-\frac{e^{t}(\frac{4e^{2t}}{(1+e^{2t})^2}(C_1e^{-t}+C_2e^{t})+C_0)}{2},\\
&=&-C_1\frac{2e^{2t}}{(1+e^{2t})^2}- C_2\frac{2e^{4t}}{(1+e^{2t})^2}-\frac{C_0}{2}e^t,\\\notag
u'_2(t)&=&\frac{e^{-t}(\frac{4e^{2t}}{(1+e^{2t})^2}(C_1e^{-t}+C_2e^{t})+C_0)}{2}\\
&=&\frac{2C_1}{(1+e^{2t})^2}+ C_2\frac{2e^{2t}}{(1+e^{2t})^2}+\frac{C_0}{2}e^{-t}.
\end{eqnarray}
Integrating these first order ODEs we obtain
\begin{eqnarray}
u_1(t) &=&\frac{C_1}{1+e^{2t}}- \frac{C_2}{1+e^{2t}}-C_2\ln (1+e^{2t})-\frac{C_0}{2}e^t,\\
u_2(t) &=&2C_1t+\frac{C_1}{1+e^{2t}} -C_1\ln (1+e^{2t}) - \frac{C_2}{1+e^{2t}}-\frac{C_0}{2}e^{-t}.
\end{eqnarray}
Substituting these into (\ref{VP}) we obtain the general solution of (\ref{GD10}) as
\begin{equation}
\rho(t)=(C_1-C_2+C_3)e^{-t}+(2C_1t+C_4)e^{t}-(C_1e^t+C_2e^{-t})\ln (1+e^{2t})-C_0.
\end{equation}
Noting that $t=\ln|\tan \frac{r}{2}|$ we have
\begin{eqnarray}\notag
\rho(r)&=&(C_1-C_2+C_3)|\cot \frac{r}{2}|+(2C_1\ln|\tan \frac{r}{2}|+C_4)|\tan \frac{r}{2}|\\&&-(C_1|\tan \frac{r}{2}|+C_2|\cot \frac{r}{2}|)\ln (1+\tan^2 \frac{r}{2})-C_0,
\end{eqnarray}
where $C_1, C_2, C_3, C_4$ are constants. This completes the proof of the theorem.
\end{proof}
\begin{example}
The map $\varphi:(S^2,dr^2+\sin^2(r)d\theta^2)\longrightarrow
(N^2,d\rho^2+(\rho^2+2C_0\rho+C)d\phi^2)$ with $ C>0$ and $\varphi(r,\theta)=(|\cot \frac{r}{2}|[1+\ln (1+\tan^2 \frac{r}{2})], \theta)$ is a proper biharmonic map. This is obtained from Theorem \ref{mv} with $C_0=C_1=C_3=C_4=0, C_2=-1$ and hence (\ref{GD26}) becomes $\rho(r)=|\cot \frac{r}{2}|[1+\ln (1+\tan^2 \frac{r}{2})]$. 
\end{example}
\begin{remark}
(i) Note that the stereographic projections $\phi: (S^2\setminus\{N\},  dr^2+\sin^2(r)d\theta^2)\longrightarrow
(r^2,d\rho^2+\rho^2d\phi^2)$ with $\phi (r, \theta)=(\cot \frac{r}{2}, \theta)$ and $\phi: (S^2\setminus\{S\},  dr^2+\sin^2(r)d\theta^2)\longrightarrow
(r^2,d\rho^2+\rho^2d\phi^2)$ with $\phi (r, \theta)=(\tan \frac{r}{2}, \theta)$ are among the maps in the family provided by Theorem \ref{mv}. It is well known that these are harmonic maps.\\
(ii) We notice that the solution does not depend on $C$ in the prescribed metric $d\rho^2+(\rho^2+2C_0\rho+C) d\phi^2$ on the target manifold. One can easily check that the Gauss curvature of the metric $d\rho^2+(\rho^2+2C_0\rho+C) d\phi^2$ is given by $K= \frac {C_0^2-C}{(\rho^2+2C_0\rho+C)^2}$. This allows us to construct examples of local proper biharmonic maps from a $2$-sphere into a surface with curvature of any fixed sign by a suitable choice of $C$. \\
(iii) Note that none of the locally defined proper biharmonic maps from $S^2$ given in Theorem \ref{mv} can be extended to a global map $\varphi:(S^2,dr^2+\sin^2(r)d\theta^2)\longrightarrow
(N^2, d\rho^2+(\rho^2+2C_0\rho+C) d\phi^2)$. If otherwise, we could choose $ C$ so that $C>C_0^2$ and hence the Gauss curvature of the target surface would be negative as we mentioned in (i). This would  contradict a theorem of  Jiang stating that any biharmonic map from a compact manifold into a non-positively curved manifold has to be harmonic.
\end{remark}
As a closing remark, we would like to point out that our results (Theorems \ref{ow111}, \ref{pb}, \ref{mv}, Corollary \ref{S2}, and Remarks 4, 5, 6) seem to suggest the following\\

{\em Conjecture:} any biharmonic map $S^2\longrightarrow (N^n,h)$ is a  weakly conformal immersion.

\end{document}